\documentclass[11pt]{amsproc}

\usepackage{url}

\usepackage[backrefs]{amsrefs}

\usepackage{enumerate}% for roman enumerate style
\usepackage{csquotes}% for quotes

\usepackage{xargs}
\usepackage{amsmath,amssymb,amsfonts,amsthm}
\usepackage{mathtools}

\usepackage{extpfeil}% maybe delete
\usepackage{tikz-cd}

\newcommandx{\newnotion}[2][2=\empty]{\emph{#1}\ifx#2\empty\index{#1}\else\index{#2}\fi}

\newcommand{\defeq}{\coloneqq}
\newcommandx{\set}[2][2=\empty]{\{#1\ifx#2\empty\else\mid#2\fi\}}% set builder notation
\newcommandx{\seq}[2][2=\empty]{{(#1\ifx#2\empty\else:#2\fi)}}% sequence builder notation
\newcommand{\subspeq}{\leq}% first space is contained in second space
\newcommandx{\gensubsp}[2][2=\empty]{\langle#1\ifx#2\empty\else\mid#2\fi\rangle}% generated subspace
\newcommand{\finfield}{\mathbb{F}}% finite field
\newcommand{\rest}[1]{\left.#1\right\rvert} %restrict
\newcommand{\setmeet}{\cap}% meet two sets
\newcommandx{\gensubgrp}[2][2=\empty]{\langle#1\ifx#2\empty\else\mid#2\fi\rangle}% generated subgroup
\DeclareMathOperator{\SU}{SU}% special unitary group
\DeclareMathOperator{\Sp}{Sp}% special unitary group
\newcommand{\centersubgrp}{\mathbf{Z}}% centre of a group
\DeclareMathOperator{\id}{id}% identity
\DeclareMathOperator{\rk}{rk}% rank of a matrix
\newcommand{\nats}{\mathbb{N}}% natural numbers
\newcommand{\card}[1]{\lvert#1\rvert}% cardinality
\DeclareMathOperator{\SL}{SL}% special linear group
\DeclareMathOperator{\PSL}{PSL}% projective special linear group
\DeclareMathOperator{\diag}{diag} % diagonal matrix
\newcommand{\setjoin}{\cup}% join two sets
\newcommand{\nathom}[1]{\overline{#1}}% natural
\newcommand{\subgrpeq}{\leq}% first group is contained in second group
\newcommand{\iso}{\cong}

\DeclareMathOperator{\rchar}{char}
\newcommandx{\gennorsubgrp}[2][2=\empty]{\langle\!\langle#1\ifx#2\empty\else\mid#2\fi\rangle\!\rangle}% generated normal subgroup
\newcommand{\commutator}[1]{[#1]}% commutator of group elements
\newcommand{\norsubgrpeq}{\trianglelefteq}% first group is normal in second group
\DeclareMathOperator{\GI}{GI}% general isometry group of a form
\DeclareMathOperator{\ct}{ct}% convergence type
\newcommand{\floor}[1]{\left\lfloor#1\right\rfloor}% floor function
\DeclareMathOperator{\GO}{GO}% feneral orthogonal gro
\newcommand{\trivgrp}{\mathbf{1}}
\DeclareMathOperator{\SO}{SO}% special orthogonal group
\DeclareMathOperator{\supp}{supp}
\DeclareMathOperator{\Alt}{Alt}

\theoremstyle{plain}
\newtheorem{lemma}{Lemma}
\newtheorem{theorem}{Theorem}

\newtheorem{corollary}{Corollary}

\theoremstyle{definition}
\newtheorem{remark}{Remark}
\newtheorem{example}{Example}

\title[On the normal subgroup lattice]{A note on the normal subgroup lattice of ultraproducts of finite quasisimple groups}

\author{Jakob Schneider}
\address{J.~Schneider, TU Dresden, 01062 Dresden, Germany}
\email{jakob.schneider@tu-dresden.de}

\author{Andreas Thom}
\address{A.~Thom, TU Dresden, 01062 Dresden, Germany}
\email{andreas.thom@tu-dresden.de}

\begin{document}
	\begin{abstract}
		In \cite{stolzthom2014lattice} it was stated that the lattice of normal subgroups of an ultraproduct of finite simple groups is always linearly ordered. This is false in this form in most cases for classical groups of Lie type. We correct the statement and point out a version of \enquote*{relative} bounded generation results for classical quasisimple groups and its implications on the structure of the lattice of normal subgroups of an ultraproduct of quasisimple groups.
	\end{abstract}
	\maketitle
	
	\tableofcontents
	
	\section{Introduction}
	
	The purpose of this note is to correct statements from the work of Stolz and the second author in \cite{stolzthom2014lattice} and generalize to a setting of quasisimple groups. As stated, \cite[Theorem~3.9]{stolzthom2014lattice} is not correct and our main result (Theorem~\ref{thm:main}) should replace it. Note that already in \cite{dowerkthom2015bounded}, it was pointed out that some of the techniques and results of \cite{stolzthom2014lattice} were flawed. Some corrections on results about bounded generation in the setting of unitary groups on finite-dimensional Hilbert spaces can be found in \cite{dowerkthom2015bounded}. The statement of \cite[Theorem~4.20]{stolzthom2014lattice} should be considered as open problem at the moment. In this note we focus entirely on the case of finite groups.

\vspace{0.1cm}

Let $G$ be a group, $g \in G$ and $L \subset G$. We write $g^G := \{ hgh^{-1} \mid h \in G \}$ for the conjugacy class of $g$ and call $N$ normal if $g \in L$ implies $g^G \subset L$. We denote by $L^{\ast k}$ the set of $k$-fold products of elements in $L$, i.e. $L^{\ast k} = \{l_1\cdots l_k \mid l_1,\dots,l_k \in L \}$.
		Using the results of \cite{liebeckshalev2001diameters} it is a simple matter to prove the following result about bounded \enquote*{relative} generation for the alternating groups $A_n$ with $n\geq 5$:
	
	\begin{lemma}
		There exists $c>0$ such that for any $S,T\subseteq A_n$ normal subsets with $\card{S},\card{T}\geq 1$, $T\neq\set{1}$ for any integer $$k\geq c\max\set{\log\card{S}/\log\card{T},1}$$ it holds that $S\subseteq (T^{A_n})^{\ast k}$.
	\end{lemma}
	
	A straightforward consequence of this fact is that for any $\sigma,\tau\in A_n$ either $\sigma\in(\tau^{A_n})^{\ast k}$ or $\tau\in(\sigma^{A_n})^{\ast k}$ for an integer $k\geq c$ which easily implies that the normal subgroups of an algebraic ultraproduct of alternating groups are linearly ordered.
	
	However, in the case of finite classical simple groups of Lie type, all of the above is false in general. The prototype of a counterexample is given by the following families of elements. 
	
	\begin{example}\label{exl:non_corr_elmts}
		Let $h_1,h_2\in H\defeq\SL_q(q)\iso\PSL_q(q)$ be the elements given by $h_1=\diag(1,\lambda,\ldots,\lambda)$ and $h_2=\diag(1,\mu,\ldots,\mu)$ where $\lambda,\mu\in\finfield_q^\ast$ are arbitrary such that $\lambda\not\in\gensubgrp{\mu}$ and $\mu\not\in\gensubgrp{\lambda}$, e.g.~take $\lambda=\zeta^a$ and $\mu=\zeta^b$ where $\gensubgrp{\zeta}=\finfield_q^\ast$ and $a,b>1$ are coprime with $q-1=ab$. Then it is easy to show by induction that each $h \in (h_1^H\setjoin (h_1^{-1})^H)^{\ast k}$ has at least $q-k$ non-trivial eigenvalues in $\gensubgrp{\lambda}$. From the assumptions it follows that $h_2\not\in(h_1^H\setjoin (h_1^{-1})^H)^{\ast k}$ for any such $k\leq q-1$. Similarly, it follows that 
$h_1\not\in(h_2^H\setjoin (h_2^{-1})^H)^{\ast k}$ for $k \leq q-1$.
	\end{example}
	
	This example already implies that the normal subgroups of an algebraic ultraproduct of finite simple groups of type $\SL_q(q)$ are not linearly ordered. In this note, we shall prove the \enquote*{best possible} result on bounded relative generation in classical quasisimple groups of Lie type. Recall, a quasisimple group is a perfect central extension of a simple group. Every finite non-abelian simple group has a universal quasisimple extension which is called the Schur covering group. We will also describe the lattice of normal subgroups of an algebraic ultraproduct of the universal finite quasisimple groups. 
	
	In the study of ultraproducts of finite quasisimple groups, we may ignore exceptional finite non-abelian simple groups, abelian groups, and families of bounded rank. Indeed, in those cases the ultraproduct yields a finite group or a quasisimple group defined over a pseudofinite field. Those groups and their lattice of normal subgroups are much easier to analyse. Assuming the classification of non-abelian finite simple groups, it follows that the remaining cases consist only of ultraproducts of the covering groups of alternating groups with increasing rank and of ultraproducts of quasisimple groups of Lie type with increasing rank. We will show (see Remark \ref{end}) that the lattice of normal subgroups in the case of universal covers of alternating groups is still a linear order. It remains to study ultraproducts of universal quasisimple finite groups of Lie type. Each of the groups of Lie type appearing as factors in the ultraproduct is endowed with its natural rank function $\ell^{\rm rk}$ and projective rank function $\ell^{\rm pr}$, as discussed later in the paper. We will prove the following theorem.
	
	\begin{theorem} \label{thm:main}
		Let $\seq{H_i}_{i\in I}$ be a sequence of universal quasisimple finite groups of Lie type of increasing rank. Set $G:= \prod_i H_i$ and define subgroups $N^{\rm rk}\defeq\set{\seq{g_i}\in G}[\lim_{\mathcal U}{\ell^{\rm rk}(g_i)}=0]$ and $N^{\rm pr}\defeq\set{\seq{g_i}\in G}[\lim_{\mathcal U}{\ell^{\rm pr}(g_i)}=0]$ of $G$. We also define $N_0\defeq\set{\seq{g_i}\in G}[g_i=1_{H_i}\text{ along }\mathcal U]$ and suitable subgroups $N_1,A_1 \leq G$, defined in Section \ref{sec:lat}. Then,
		\begin{enumerate}[(i)]
			\item The subgroup $N^{\rm pr}$ contains all proper normal subgroups of $G$ containing $N_0$. Especially, $G/N^{\rm pr}$ is non-abelian simple and $N^{\rm pr}/N_0$ is a characteristic subgroup of $G/N_0$.
			\item The normal subgroups of $G$ lying between $N_0$ and $N^{\rm rk}$ are linearly ordered. Moreover, each such normal subgroup is perfect.
			\item Define mappings between the following two sets
			$$
			\begin{tikzcd}
			& \set{N\norsubgrpeq G}[N_1\subgrpeq N\subgrpeq N^{\rm pr}] \arrow[d,xshift=0.7ex,"\Phi"] &\\
			& \set{M}[M\norsubgrpeq G,\ N_1\subgrpeq M\subgrpeq N^{\rm rk}]\times\set{A}[A\norsubgrpeq G,\ N_1\subgrpeq A\subgrpeq A_1] \arrow[u,xshift=-0.7ex,"\Psi"] &
			\end{tikzcd}
			$$
			by $\Phi:N\mapsto(N\setmeet N^{\rm rk}, N\setmeet A_1)$ and $\Psi:(M,A)\mapsto MA$.	Then $\Phi$ and $\Psi$ are isomorphisms of posets and mutually inverse to each other.
		\end{enumerate}
	\end{theorem}
	
	Throughout, we use the notation and basic facts from \cite{wilson2009finite}.
	
	\section{Auxiliary geometric results}\label{sec:gmtry}
	
	In this section we provide the necessary geometric results for the rest of the article. In the following let $H$ be a quasisimple group from the following list:
	\begin{enumerate}[(i)]
		\item linear: $\SL_n(q)$, $n\geq 2$, $(n,q)\neq(2,2),(2,3)$;
		\item symplectic: $\Sp_{2m}(q)$, $m\geq 2$, $(m,q)\neq(2,2)$;
		\item unitary: $\SU_n(q)$, $n\geq 3$, $(n,q)\neq(3,2)$;
		\item orthogonal: $\Omega_{2m+1}(q)$, $m\geq 3$, $q$ odd; $\Omega^\pm_{2m}(q)$, $m\geq 4$.
	\end{enumerate}
	
		Here we omit the groups $\Omega_{2m+1}(q)$ with $q$ even as they are isomorphic to the groups $\Sp_{2m}(q)$. We have also omitted to mention the spin groups, i.e. the double covers of $\Omega^{\pm}_{2n}(q)$ and $\Omega_{2n+1}(q)$, as they can be treated directly in a way similar to the double cover of the alternating groups, see Remark \ref{end}.
	
\vspace{0.1cm}	
	
	Let $V$ be the natural representation of $H$ (there are two such representations when $H=\Omega_{2n+1}(q)$ for $q$ odd, corresponding to the two equivalence classes of non-degenerate symmetric bilinear forms in this dimension) and $n\defeq\dim V$. For (ii)--(iv), denote by $f$ the corresponding non-degenerate alternating, conjugate-symmetric sesquilinear, or symmetric bilinear form. In (iv) for $q$ even, denote by $Q$ the corresponding quadratic form inducing the non-degenerate alternating form $f$.
	
	A subspace $U$ of $V$ is called non-singular if $\rest{f}_U$ is non-degenerate. In this case $U\oplus U^\perp=V$.
		The following fact will be used in the Section~\ref{sec:lat}.
	
	\begin{lemma}\label{lem:inv_non-sing_subsp}
		If $U\subspeq V$ is a non-singular subspace with $2\leq\dim U<n/2$, then there exists a decomposition $W_1\oplus W_2=U^\perp$ and an element $h\in H$ such that $h$ is the identity on $W_2$ and interchanges $U$ and $W_1$.
	\end{lemma}
	
	\begin{proof}
		Explicit computations with standard bases show that $U^\perp$ always contains a subspace $W_1$ isometric to $U$ via an isometry $\theta:U\to W_1$ with respect to $f$ or $Q$ in the orthogonal case in characteristic two (here we need $\dim U^\perp>\dim U$). Set $W_2\defeq(U\oplus W_1)^\perp$. Then $h_1:V=U\oplus W_1\oplus W_2\to U\oplus W_1\oplus W_2=V$ given by $(u,w_1,w_2)\mapsto(\theta^{-1}(w_1),\theta(u),w_2)$ is an isometric involution of $V$. 

Hence, in the symplectic case (ii) we may take $h\defeq h_1$. 
In the unitary case (iii), if $\det(h_1)=\varepsilon\in\set{\pm 1}$, letting $h_2$ be the linear map which scales a non-isotropic vector of $U$ by $\varepsilon$ and fixes its perpendicular complement, then $h\defeq h_1 h_2$ works.
		
		In the orthogonal case (iv) for $q$ odd, define $h_2$ as in the unitary case. If the spinor norm of $h_1 h_2$ is $\varepsilon\in\set{\pm 1}$, find an element $s\in\SO(U)$ of spinor norm $\varepsilon$ and set $h_3\defeq s\oplus\id_{U^\perp}$ (for the existence of $s$ we use that $\dim U\geq 2$). Then $h\defeq h_1 h_2 h_3$ works. 
		
		In the orthogonal case for $q$ even (forcing $\dim V$ to be even), if the quasideterminant of $h_1$ is $\varepsilon\in\set{\pm 1}$, find an element $s\in\GO(U)$ of quasideterminant $\varepsilon$ and set $h_4\defeq s\oplus\id_{U^\perp}$ (again using $\dim U\geq 2$). Then $h\defeq h_1 h_4$ is an appropriate choice. The proof is complete.
	\end{proof}
	
	The following simple result about the existence of certain non-singular subspaces will also be used in later sections.
	
	\begin{lemma}\label{lem:subsp_to_nonsing_subsp}
		Let $U\subspeq V$, $\dim U=l$. Then there exists $W\subspeq U$ non-singular such that $\dim W\geq 2l-n$.
	\end{lemma}
	
	\begin{proof}
		Chose $W\subspeq U$ maximal non-singular. If there were two vectors $u,v\in W^\perp\setmeet U$ with $f(u,v)\neq 0$, then $W\oplus\gensubsp{u,v}>W$ would still be non-singular. Hence $f$ is zero on this subspace. Moreover, for dimension reasons $U=W\oplus(W^\perp\setmeet U)$. Together this implies $U^\perp\setmeet U=W^\perp\setmeet U$. Hence $$\dim(W^\perp\setmeet U)=\dim U-\dim W\leq\dim U^\perp=n-l,$$ so $\dim W\geq 2l-n$ as desired.
	\end{proof}
	
	We introduce the group of quasiscalars $S(H)$ of $H$ by
	\begin{align*}
	S(\SL_n(q))&\defeq \finfield_q^\ast\iso C_{q-1};\\ S(\SU_n(q))&\defeq\set{x\in\finfield_{q^2}^\ast}[x^{q+1}=1]\iso C_{q+1};\\
	S(\Sp_{2m})&\defeq\set{\pm 1};\\
	S(\Omega_{2m+1}(q))&\defeq\set{\pm 1} \quad (q\text{ odd});\\
	S(\Omega_{2m}^{\pm}(q))&\defeq\set{\pm 1}.
	\end{align*}
	Of course $-1=1$ if $\rchar(\finfield_q)=2$.
		Our last auxiliary result will be used in Section~\ref{sec:lat}:
	
	\begin{lemma}\label{lem:all_qsi_scal_attained}
		For any $\lambda\in S(H)$ there is a diagonalizable element in $h\in H$ with respect to $V$ and a suitable basis such that all but two of its diagonal entries are equal to $\lambda$.
	\end{lemma}
	
	\begin{proof}
		Clearly, we may assume that $\lambda\neq 1$, since otherwise we can always take $h=\id$. If $H=\SL_n(q)$, take $h=\diag(\lambda,\ldots,\lambda,\lambda^{-(n-1)})$ with respect to any basis of $V$. If $H=\Sp_{2m}(q)$, $q$ odd, and $\lambda=-1$, take $-\id_V$. If $H=\SU_n(q)$, take $h=\diag(\lambda,\ldots,\lambda,\lambda^{-(n-1)})$ with respect to an orthonormal basis $e_1,\ldots,e_n$ for $f$, i.e.~$f(e_i,e_j)=\delta_{ij}$ for $i,j=1,\ldots,n$.
		
		For $\Omega_{2m+1}(q)$, $q$ odd, and $\lambda=-1$, take $\diag(-1,\ldots,-1,1)$ with respect to a basis $e_1,\ldots,e_{2m+1}$ such that $f(e_i,e_j)=0$ if $i\neq j$ and $f(e_i,e_i)=1$ for $i=1,\ldots,2m$ and $f(e_{2m+1},e_{2m+1})=1$ or a non-square $\alpha\in\finfield_q^\ast$ (there are two equivalence classes of such forms both giving a natural representation of $\Omega_{2n+1}(q)$). Clearly $h$ has determinant and spinor norm equal to one.
		
		For $H=\Omega_{2n}^{\pm}(q)$, $q$ odd, and $\lambda=-1$, we find a basis $e_1,\ldots,e_{2n}$ such that either $f(e_i,e_j)=\delta_{ij}$ for all $i,j=1,\ldots,2n$, or $f(e_i,e_j)=0$ for $i\neq j$, $f(e_i,e_i)=1$ for $i=1,\ldots,2n-1$ and $f(e_{2n},e_{2n})=\alpha\in\finfield_q^\ast$ a non-square (there are two equivalence classes of non-degenerate symmetric bilinear forms corresponding to the two non-isomorphic groups $\Omega_{2n}^+(q)$ and $\Omega_{2n}^-(q)$).
		In either case we can take $h=\diag(-1,\ldots,-1,1,1)$ with respect to this basis, which has determinant and spinor norm one.
		
		In all remaining cases we have $S(H)=\trivgrp$, so the proof is complete.
	\end{proof}
	
	\section{Relative bounded generation in universal finite quasisimple groups}\label{sec:rel_bd_gen}
	
	In this section we keep the notation from Section~\ref{sec:gmtry}. In addition, we define the following three length functions on our group $H$: The rank length function $\ell^{\rm rk}(h)\defeq\rk(1-h)/n$, the projective rank length function $\ell^{\rm pr}(h)\defeq\min\set{\ell^{\rm rk}(\lambda^{-1}h)}[\lambda\in\finfield_q^\ast\text{ or }\finfield_{q^2}^\ast]$ ($q^2$ when $H$ is unitary), and the conjugacy length function $\ell^{\rm c}(h)\defeq\log\card{h^H}/\log\card{H}$.
	
	In the following we will prove a version of \enquote*{relative bounded generation} for all universal quasisimple groups from families of unbounded rank (for the others there is no such notion). We start with the alternating case. Recall that the normalized Hamming length defined by $\ell^{\rm H}(\sigma)\defeq\card{\supp(\sigma)}/n$ for $\sigma\in A_n$ is Lipschitz equivalent to $\ell^{\rm c}$ by \cite[Theorem~2.11]{stolzthom2014lattice}. We denote the universal central extension of $A_n$ by $\widetilde{A}_n$.

	\begin{lemma}\label{lem:rel_bd_gen_alt_case}
		There exists a constant $c>0$ such that for any $\sigma\in\widetilde{A}_n$, $\tau\not\in\centersubgrp(\widetilde{A}_n)$ for any integer $k\geq c\max\set{\ell^{\rm H}(\nathom{\sigma})/\ell^{\rm H}(\nathom{\tau}),1}$ we have $\sigma\in(\tau^{\widetilde{A}_n})^{\ast k}$. Here $\nathom{\sigma},\nathom{\tau}\in A_n$ are the images of $\sigma,\tau\in\widetilde{A}_n$ under the canonical map.
	\end{lemma}

	\begin{proof}		
		We prove the lemma for $A_n$ ($n\geq 5$) and derive the corresponding result for its Schur covering group.
		Let $\sigma,\tau\in A_n$, $\tau\neq 1$. After conjugating $\sigma$, we may assume that either $\supp(\sigma)\subseteq\supp(\tau)$ or the opposite inclusion holds. In the first case, $\sigma\in\Alt(\supp(\tau))=(\tau^{A_n})^{\ast c_1}$ for some integral constant $c_1>0$ and in the second case $\sigma\in\Alt(\supp(\sigma))=(\tau^{A_n})^{\ast k}$ for any integer $k\geq c_2/\ell^{\rm H}_{\Alt(\supp(\sigma))}(\tau)=c_2\ell^{\rm H}(\sigma)/\ell^{\rm H}(\tau)$ with some constant $c_2>0$; both times we use \cite[Lemma~3.2]{liebeckshalev2001diameters}. So we may take $c\defeq\max\set{c_1,c_2}$.
		
		The statement now extends to the Schur cover $\widetilde{A}_n$ by the following argument. Without loss of generality, we may assume $n>7$ (we may neglect this finite data), so $\widetilde{A}_n$ is a twofold cover.
		It follows from the standard construction of the two double covers of $S_n$ that when $a$ and $b$ are lifts of the transposition $(12)$ and $(34)$, then $(ab)^2=z$ is the unique non-trivial central element in one of these. This shows that $\centersubgrp(\widetilde{A}_n)\subseteq(\tau^{\widetilde{A}_n})^{\ast c_3}$ for any non-central $\tau\in\widetilde{A}_n$ and some absolute integer constant $c_3>0$, which implies the claim for $\widetilde{A}_n$.
	\end{proof}

	Now we turn to the universal quasisimple groups from families of unbounded rank. The proof of \enquote*{relative bounded generation} for these is actually very similar as for the Schur covering groups of the alternating groups.
	
	We need the following fact which can be deduced by adapting the proof of Lemma~4.1 of \cite{liebeckshalev2001diameters} to quasisimple groups and looking at Lemma~5.4, 6.4, and the end of Section~7 of the same article.
	
	\begin{lemma}\label{lem:bd_gen}
		There is an absolute constant $c>0$ (independent of $H$) such that for $h\not\in H\setminus Z(H)$, for any $k\geq c/\ell^{\rm pr}(h)$ it holds that $H=(h^H)^{\ast k}$.
	\end{lemma}
	
	\begin{remark}
		It follows from Lemma~5.3, 6.3, and the end of Section~7 of \cite{liebeckshalev2001diameters} that $\ell^{\rm pr}$ and $\ell^{\rm c}$ are Lipschitz equivalent by a universal Lipschitz constant, so putting $\ell^{\rm c}$ in place of $\ell^{\rm pr}$ in the Lemma~\ref{lem:bd_gen} would also yield a valid statement.
	\end{remark}
	
	Here is now the promised result for almost all other universal quasisimple groups from families of unbounded rank:
	
	\begin{lemma}\label{lem:rk_bd_gen_class_grps}
		Let $\varepsilon>0$ be arbitrary. There exists an absolute constant $C>0$ and a constant $D>0$ only depending on $\varepsilon$ such that the following holds:
		Let $h_1\in H\setminus\centersubgrp(H)$, $h_2\in H$, $\ell^{\rm rk}(h_1)\leq 1-\varepsilon$.
		Then $h_2\in(h_1^H)^{\ast k}$ for all integers $k\geq\max\set{C\ell^{\rm rk}(h_2)/\ell^{\rm rk}(h_1),D}$.
	\end{lemma}
	
	\begin{proof}
		First assume that $\varepsilon\leq\ell^{\rm rk}(h_1)\leq 1-\varepsilon$. Then by Proposition 2.13 of \cite{stolzthom2014lattice} it holds that $\ell^{\rm pr}(h_1)\geq\varepsilon$, so by Lemma~\ref{lem:bd_gen} there is $D\in\nats$ only depending on $\varepsilon$ such that $(h_1^H)^{\ast D}=H$.
		
		So we may assume w.l.o.g.~that $\ell^{\rm rk}(h_1)<\varepsilon\leq 1/8$. Assume additionally that $\ell^{\rm rk}(h_2)\leq 1/8$ as well (we will remove this assumption at the end). Set $U_i\defeq\ker(1-h_i)\subspeq V$ and $l_i\defeq\dim U_i$ ($i=1,2$). 
		
		At first we treat the special linear case, i.e.~$H=\SL_n(q)$: Replace $h_1,h_2\in H$ by conjugates in Jordan normal form (where one off-diagonal entry might be different from $0$ or $1$) as matrices with respect to a suitable basis $e_1,\ldots,e_n$ with all $1\times 1$ Jordan blocks corresponding to eigenvalue one in the upper left corner. Assuming there are $m_i$ of these for $h_i$, i.e.~$h_i e_1=e_1,\ldots,h_i e_{m_i}=e_{m_i}$, it is easy to see that $m_i\geq 2l_i-n$ ($i=1,2$). Set $m\defeq\min\set{m_1,m_2}$ and $W'\defeq\gensubsp{e_{m+1},\ldots,e_n}$. Choose $n-m\leq j\leq m$ minimal such that for $X\defeq\gensubsp{e_{m-j+1},\ldots,e_m}$, the space $Y\defeq X\oplus W'$ gives rise to a quasisimple group $K\defeq\SL(Y)$. This is possible since by assumption on $l_1$ and $l_2$ we have $m\geq 3n/4>n/2$ and $\SL(V)$ is quasisimple. Since $\dim W'\geq 1$, there is an absolute constant $d>1$ such that $j\leq d \dim W'=d(n-m)$.
		Then set $Y'\defeq\gensubsp{e_1,\ldots,e_{m-j}}$. The operators $h_1,h_2$ write as $h_i=\id_{Y'}\oplus\id_X\oplus A_i=\id_{Y'}\oplus B_i$ ($i=1,2$) with respect to te decompositions $V=Y'\oplus X\oplus W'=Y'\oplus Y$. Assuming $m_1\leq m_2$ gives $(1+d)^{-1}/2\leq\ell^{\rm rk}(B_1)\leq1/2$, so by Proposition~2.13 of \cite{stolzthom2014lattice} we have $\ell^{\rm pr}(B_1)\geq (1+d)^{-1}/2$ implying the existence of a constant $D\in\nats$ with $B_2\in(B_1^K)^{\ast D}$ which yields $h_2\in(h_1^H)^{\ast D}$ (by Lemma~\ref{lem:bd_gen}). If $m_2<m_1$ the same computation shows $(1+d)^{-1}/2<\ell^{\rm rk}(B_2)$ and $\ell^{\rm rk}(B_1)\leq1/2$, so as previously $\ell^{\rm pr}(B_1)=\ell^{\rm rk}(B_1)$. By Lemma~\ref{lem:bd_gen} there is $c>0$ such that for all integers $k\geq c/\ell^{\rm pr}(B_1)$ we have $K=(B_1^K)^{\ast k}$. Then it holds that $h_2\in(H_1^H)^{\ast k}$ and by
		$$
		2(1+d)c\ell^{\rm rk}(h_2)/\ell^{\rm rk}(h_1)=2(1+d)c\ell^{\rm rk}(B_2)/\ell^{\rm rk}(B_1)>c/\ell^{\rm pr}(B_1)
		$$
		we are done in this case with $C\defeq 2(1+d)c$.
		
		In the other cases, i.e.~$H\neq\SL_n(q)$, the proof is almost identical:
		Define $U\defeq U_1\setmeet U_2$ and $l\defeq\dim U\geq l_1+l_2-n$. From Lemma~\ref{lem:subsp_to_nonsing_subsp} we get $W\subspeq U$ non-singular with $\dim W\geq 2l-n$. Then we infer $\dim W^\perp\leq 2(n-l)$ from $\dim W\geq 2l-n$.
		This implies
		$$
		\dim W^\perp\leq 2(n-l)\leq2(2n-(l_1+l_2))\leq 4(n-\min\set{l_1,l_2})
		$$
		and
		$$
		4\min\set{l_1,l_2}-3n\leq 2(l_1+l_2)-3n\leq 2l-n\leq\dim W.
		$$
		As by assumption $\ell^{\rm rk}(h_1),\ell^{\rm rk}(h_2)\leq 1/8$, implying that $7/8n\leq l_1,l_2$, we obtain $\dim W^\perp\leq \dim W$. Now take $X\leq W$ a non-singular subspace such that $\dim X\geq\dim W^\perp$ is as small as possible such that $Y\defeq X\oplus W^\perp$ gives rise to a classical quasisimple group $K$ from the beginning of Section~\ref{sec:gmtry}. As $\dim W^\perp\geq 1$, then $\dim X\leq d\dim W^\perp$ for some absolute $d>1$. With respect to the decompositions $V=Y^\perp\oplus X\oplus W^\perp=Y^\perp\oplus Y$, the operators $h_1, h_2$ write as $h_i=\id_{Y^\perp}\oplus\id_X\oplus A_i=\id_{Y^\perp}\oplus B_i$ with isometric automorphism $A_i$ of $W^\perp$ and $B_i$ of $Y$ ($i=1,2$).
		Now assume $l_1\leq l_2$. Then
		$$
		\ell^{\rm rk}(B_1)=\frac{\rk(1-A_1)}{\dim W^\perp+\dim X},
		$$
		which can be bounded from above by $1/2$ and from below by the chain
		$$
		\frac{n-l_1}{2(1+d)(n-l)}\geq\frac{n-l_1}{2(1+d)(2n-(l_1+l_2))}\geq\frac{n-l_1}{4(1+d)(n-l_1)}=\frac{1}{4(1+d)}.
		$$
		Hence by Proposition 2.13 of \cite{stolzthom2014lattice} we get that $\ell^{\rm pr}(B_1)\geq (1+d)^{-1}/4$, so by Lemma~\ref{lem:bd_gen} there is a constant $D\in\nats$ such that $B_2\in(B_1^K)^{\ast D}$ implying that $h_2\in(h_1^H)^{\ast D}$.		
		On the other hand, if $l_2<l_1$, by the same computation as above, $0.25(1+d)^{-1}<\ell^{\rm rk}(B_2)$ and $\ell^{\rm rk}(B_1)\leq 1/2$, so by Proposition 2.13 of \cite{stolzthom2014lattice} $\ell^{\rm pr}(B_1)=\ell^{\rm rk}(B_1)$. Applying Lemma~\ref{lem:bd_gen} gives $c>0$ such that for
		all integers $k\geq c/\ell^{\rm pr}(B_1)$ we have $K=(B_1^K)^{\ast k}$. But then $h_2\in(h_1^H)^{\ast k}$ for such $k$, and since 
		$$
		4(1+d)c\ell^{\rm rk}(h_2)/\ell^{\rm rk}(h_1)=4(1+d)c\ell^{\rm rk}(B_2)/\ell^{\rm rk}(B_1)>c/\ell^{\rm pr}(B_1)
		$$ 
		we are done in this case with $C\defeq 4(1+d)c$.
		
		We still need to eliminate the condition $\ell^{\rm rk}(h_2)\leq 1/8$. But this is again just another straightforward application of Lemma~\ref{lem:bd_gen}, since the elements $h\in H$ with $\ell^{\rm rk}(h)> 1/8$ generate $H$ \enquote*{quickly}.
	\end{proof}
	
	\begin{remark}
		The condition $\ell^{\rm rk}(h_1)\leq 1-\varepsilon$ for a fixed $\epsilon>0$ cannot be removed by Example~\ref{exl:non_corr_elmts}. In that sense, the previous result is best possible.
	\end{remark}
	
	\begin{remark}\label{rmk:orth_case_dbl_cover_rel_bd_gen}
		The only universal quasisimple groups from families of unbounded rank which are not covered by Lemma~\ref{lem:rel_bd_gen_alt_case} and \ref{lem:rk_bd_gen_class_grps} are the double covers of the orthogonal groups in odd characteristic.
		
		Defining $\ell^{\rm rk}(h)\defeq\ell^{\rm rk}(\nathom{h})$ for $h$ an element of the twofold cover of $\Omega^\pm_{2m}(q)$ or $\Omega_{2m+1}(q)$ ($q$ odd), the statement of Lemma~\ref{lem:rk_bd_gen_class_grps} also holds for these. This is true since $\widetilde{A}_5$ is embedded in both of them as the group of products of an even number of the reflections with respect to the vectors $e_1-e_2, e_2-e_3, e_3-e_4, e_4-e_5\in V$, where $e_1,\ldots,e_5$ is an orthonormal system for $f$. So we can argue as at the end of the proof of Lemma~\ref{lem:rel_bd_gen_alt_case} to see that the two-element kernel of the covering map is generated \enquote*{quickly} by conjugates of any non-central element.
	\end{remark}
	
	\section{The lattice of normal subgroups of an algebraic ultraproduct of classical finite quasisimple groups}\label{sec:lat}
	
	Let $\seq{H_i}_{i\in I}$ be a sequence of groups from the list at the beginning of Section~\ref{sec:gmtry} and set $G\defeq\prod_{i\in I}{H_i}$. Let $n_i$ be the dimension of the \enquote*{natural} representation $V_i$ of $H_i$. Note that the sequence $\seq{n_i}$ is Lipschitz equivalent to the sequence of Lie ranks of the $H_i$ by a universal Lipschitz constant. For some ultrafilter $\mathcal U$ on $I$ for which $\lim_{\mathcal U}{n_i}=\infty$, define the normal subgroup $N_0\defeq\set{\seq{g_i}\in G}[g_i=1_{H_i}\text{ along }\mathcal U]$ of $G$. In this section we give a complete description of the lattice of normal subgroups of the algebraic ultraproduct $G/N_0$ of the groups $H_i$ ($i\in I$) with respect to $\mathcal U$.
	
	Define the subgroups 
	$$N^{\rm rk}\defeq\set{\seq{g_i}\in G}[\lim_{\mathcal U}{\ell^{\rm rk}(g_i)}=0] \quad \mbox{and} \quad N^{\rm pr}\defeq\set{\seq{g_i}\in G}[\lim_{\mathcal U}{\ell^{\rm pr}(g_i)}=0]
	$$ of $G$ where $\ell^{\rm rk}$ and $\ell^{\rm pr}$ are the length functions defined at the beginning of Section~\ref{sec:rel_bd_gen}. As they are invariant length functions, it is clear that $N^{\rm rk}$ and $N^{\rm pr}$ are normal in $G$ and contain $N_0$. Moreover, as $\ell^{\rm pr}\leq\ell^{\rm rk}$, we get that $N^{\rm rk}\subgrpeq N^{\rm pr}$.
	
	The following result is an immediate consequence of Lemma~\ref{lem:bd_gen}:
	
	\begin{lemma}\label{lem:el_outside_Npr}
		The subgroup $N^{\rm pr}$ contains all proper normal subgroups of $G$ containing $N_0$. Especially, $G/N^{\rm pr}$ is non-abelian simple and $N^{\rm pr}/N_0$ is a characteristic subgroup of $G/N_0$.
	\end{lemma}
	
	\begin{proof}
		When $g=\seq{g_i}\in G\setminus N^{\rm pr}$, there exist $\varepsilon>0$ and $U\in\mathcal U$ such that $\ell^{\rm pr}(g_i)\geq\varepsilon$ for $i\in U$. Hence from  Lemma~\ref{lem:bd_gen} it follows that there is $k\in\nats$ such that
		$(g_i^{H_i})^{\ast k}=H_i$ for $i\in U$. This implies $\gennorsubgrp{g}N_0=G$ as wished.
	\end{proof}
	
	Hence from now on, we may restrict to the normal subgroups above $N_0$ which are contained in $N^{\rm pr}$. Let us first characterize $N^{\rm pr}$ among the subgroups of $G$ containing $N^{\rm rk}$. For this recall the definition of $S(H)$ from the end of Section~\ref{sec:gmtry}.
	
	\begin{lemma}\label{lem:char_cent_of_G/Nrk}
		The map $\varphi:N^{\rm pr}\to Z\defeq\prod_{\mathcal U}{S(H_i)}$ defined by $\seq{g_i}\mapsto\nathom{\seq{\lambda_i}}$, where $\lambda_i$ is arbitrary (from $S(H_i)$) if $\ell^{\rm pr}(g_i)\geq 1/4$ and $\lambda_i$ is the unique $\lambda\in\finfield_{q_i}^\ast$ for which $\ell^{\rm rk}(\lambda^{-1}g_i)<1/4$ otherwise, is a surjective homomorphism with kernel $N^{\rm rk}$. Moreover, $Z\iso N^{\rm pr}/N^{\rm rk}=\centersubgrp(G/N^{\rm rk})$.
	\end{lemma}
	
	\begin{proof}
		At first we check that $\lambda_i$ is always in $S(H_i)$: This is clear when $\ell^{\rm pr}(g_i)\geq 1/4$. In the special linear case there is nothing to check, so consider the remaining cases.
		Set $U_i$ to be the eigenspace corresponding to eigenvalue $\lambda_i$ of $g_i$ and set $\dim U_i\defeq l_i>3/4n_i$. Then by Lemma~\ref{lem:subsp_to_nonsing_subsp} there is $W_i\subspeq U_i$ non-singular such that $\dim W_i\geq 2l_i-n_i>n_i/2\geq 1$. The restriction $\rest{g_i}_{W_i}$ is a scalar from the group $\GI(W_i,\rest{f_i}_{W_i})$, so it must lie in $S(H_i)$.
		
		Let us now show that $\varphi$ is a homomorphism. So let $\seq{g_i},\seq{h_i}\in N^{\rm pr}$ and pick $U\in\mathcal U$ and sequences $\seq{\lambda_i},\seq{\mu_i}$ from $\prod_i{S(H_i)}$ such that $$
		\ell^{\rm rk}(\lambda_i^{-1}g_i),\ell^{\rm rk}(\mu_i^{-1}h_i)<1/8
		$$ 
		for $i\in U$ (this is possible by the definition of $N^{\rm pr}$). Then it follows that $\ell^{\rm rk}((\lambda_i\mu_i)^{-1}g_ih_i)<1/4$ for $i\in U$ by the triangle inequality, establishing that $\varphi$ is a homomorphism.
		
		Also $\varphi$ is surjective: Let $\lambda=\nathom{\seq{\lambda_i}}\in Z$. In any case, by Lemma~\ref{lem:all_qsi_scal_attained} there is a diagonal matrix in $H_i$ with all but at most two entries of the diagonal equal to $\lambda_i$. This is the desired preimage of $\lambda$.
		
		Moreover, the kernel of $\varphi$ consists of sequences $g=\seq{g_i}$ such that when $\nathom{\seq{\lambda_i}}$ is the image of $g$ under $\varphi$, then $\lambda_i=1$ for $i\in U$ and some $U\in\mathcal U$. But this means precisely that $g\in N^{\rm rk}$, since then $0=\lim_{\mathcal U}{\ell^{\rm pr}(g_i)}=\lim_{\mathcal U}{\ell^{\rm rk}(\lambda_i^{-1}g_i)}=\lim_{\mathcal U}{\ell^{\rm rk}(g_i)}$.
		
		Lastly, we have to show that $N^{\rm pr}/N^{\rm rk}$ is the center of $G/N^{\rm rk}$. Since by Lemma~\ref{lem:el_outside_Npr} the quotient $G/N^{\rm pr}$ is non-abelian simple, it suffices to show that $\commutator{G,N^{\rm pr}}\subseteq N^{\rm rk}$. So let $g=\seq{g_i}\in G$ and $h=\seq{h_i}\in N^{\rm pr}$ and let $\seq{\lambda_i}$ be sequence in $\finfield_{q_i}^\ast$ such that $\seq{\lambda_i^{-1}h_i}\in N^{\rm rk}$. Let $U_i$ be the eigenspace of $h_i$ corresponding to eigenvalue $\lambda_i$ and set $W_i\defeq U_i\setmeet g_i^{-1} U_i$. Then for $w\in W_i$ one computes that $\commutator{g_i,h_i}w=g_i^{-1}h_i^{-1}g_i h_i w=w$. As $\dim W_i/n_i$ tends to one along $\mathcal U$, this verifies that $\commutator{g,h}\in N^{\rm rk}$.	The proof is complete.
	\end{proof}
	
	To describe the normal subgroups $N$ of $N^{\rm pr}$ containing $N_0$ accurately, we need some definitions:
	Define $\mathcal L$ to be the set of null sequences $\seq{r_i}$ along $\mathcal U$ with $r_i\in[0,1]$ ($i\in I$). Two such sequences $\seq{r_i}$ and $\seq{s_i}$ are said to be equivalent (write $\seq{r_i}\sim\seq{s_i}$ and $\nathom{\seq{r_i}}$ for the corresponding equivalence class) if $\lim_{\mathcal U}{r_i/s_i}\in(0,\infty)$. For two elements $r$ and $s$ of the quotient $\mathcal L/\mathord\sim$ write $r\leq s$ if an only if $\lim_{\mathcal U}{r_i/s_i}<\infty$, where $\seq{r_i}$ and $\seq{s_i}$ are representatives for $r$ and $s$, respectively. It is routine to check that this is a definition and turns $(\mathcal L/\mathord\sim,\leq)$ into a linear order.
	
	Now define the function $\ct:N^{\rm pr}\to\mathcal L/\mathord\sim$ by $g=\seq{g_i}\mapsto\nathom{\seq{\ell^{\rm pr}(g_i)}}$ and call it the \newnotion{convergence type} of $g$. Denote by $L$ the subset of elements $r$ of $\mathcal L/\mathord\sim$ for which either $r\geq\nathom{\seq{1/n_i}}$ or $r=\nathom{\seq{0}}$.
	
	\begin{lemma}\label{lem:ct_Npr_onto_L}
		The image of the function $\ct:N^{\rm pr}\to\mathcal L/\mathord\sim$ is equal to $L$.
	\end{lemma}
	
	\begin{proof}
		At first we prove that the image of $\ct$ lies in $L$: Namely when $r=\nathom{\seq{r_i}}\in N^{\rm pr}$ and $r\neq\nathom{\seq{0}}$, then $r_i\neq 0$ for all $i\in U$ for some $U\in\mathcal U$. But then it follows that $r_i\geq 1/n_i$ for these $i$ and hence $r\geq\nathom{\seq{1/n_i}}$.
		
		The surjectivity is a bit more subtle. We prove it here only for the case that $H_i=\SL_{n_i}(q_i)$; the other cases are analogous. Choose a non-trivial element $g\in\SL_2(q)$. Then if $r=\nathom{\seq{r_i}}$ satisfies $r_i\geq 1/n_i$ along $\mathcal U$, set $g_i\defeq g^{\oplus \floor{n_i r_i}}\oplus\id_{n_i-2\floor{n_i r_i}}$ if $n_i\geq 2\floor{n_i r_i}$ and choose an arbitrary element otherwise. Letting $g=\seq{g_i}$, one sees immediately that $\ct(g)=r$. Finally, $\ct(1_G)=\nathom{\seq{0}}$ completes the proof.
	\end{proof}
	
	\begin{remark}
		The rank length function $\ell^{\rm rk}$ does not attain any possible value on the classical quasisimple groups even for large ranks: E.g.~$2n\ell^{\rm rk}(g)$ for $g\in\GO^{\varepsilon}_{2n}(q)$ is even if and only if $g\in\Omega^\varepsilon_{2n}$ (see \cite[p.~77]{wilson2009finite} at the end of Section~3.8.1).
	\end{remark}
	
	Due to the previous lemma, henceforth we shall consider $\ct$ as a function with domain $N^{\rm pr}$ and codomain $L$. We extend this function to subsets $S\subseteq N^{\rm pr}$ by setting $\ct(S)\defeq\set{\ct(g)}[g\in S]$. For a normal subgroup $N$ of $G$ between $N_0$ and $N^{\rm pr}$ the set $\ct(N)$ is then called the associated order ideal (see Lemma~\ref{lem:nor_subgrp_blw_Nrk} below).
	
	\begin{lemma}\label{lem:ct_Npr_to_Nrk}
		For such a subgroup it holds that $\ct(N)=\ct(N\setmeet N^{\rm rk})$.
	\end{lemma}
	
	\begin{proof}
		It is clear that $\ct(N)\supseteq\ct(N\setmeet N^{\rm rk})$. Let us prove the converse containment to finish the proof: Let $g=\seq{g_i}\in N\subgrpeq N^{\rm pr}$. Then there is $U\in\mathcal U$ and $\lambda_i\in\finfield_{q_i}^\ast$ ($i\in I$) such that $\ell^{\rm pr}(g_i)=\ell^{\rm rk}(\lambda_i^{-1}g_i)<1/4$ for $i\in U$ and so when $U_i$ is the eigenspace to $\lambda_i$ for such an $i$, then $l_i\defeq\dim U_i>3/4n_i$. By Lemma~\ref{lem:subsp_to_nonsing_subsp} there is a non-singular subspace $W_i\subspeq U_i$ with $\dim W_i\geq 2l_i-n_i> n_i/2$. Moreover, for all large $n_i$ we may also assume that $n_i-\dim W_i=\dim W_i^\perp\geq 2$ (by modifying $W_i$ a little). 
		Now Lemma~\ref{lem:inv_non-sing_subsp} implies the existence of $h_i\in H_i$ ($i\in I$) and $W_i^1,W_i^2\subspeq W_i$ ($i\in U$) non-singular such that $W_i^\perp\oplus W_i^1\oplus W_i^2$ is a perpendicular direct sum and $h_i$ restricts to the identity on $W_i^2$ while it interchanges $W_i^\perp$ and $W_i^1$. Then it holds that $\commutator{g_i,h_i}(h_i^{-1}w)=g_i^{-1}h_i^{-1}g_iw=\lambda_i^{-1}h_i^{-1}g_iw=\lambda_i^{-1}g_i^{h_i}(h_i^{-1}w)$ and $\commutator{g_i,h_i}w=g_i^{-1}h_i^{-1}g_ih_iw=\lambda_i g_i^{-1}w$ for $w\in W_i^\perp$ and $i\in U$. This implies that $\ell^{\rm rk}(\commutator{g_i,h_i})=2\ell^{\rm pr}(g_i)<1/2$. But then by Proposition~2.13 of \cite{stolzthom2014lattice} we have $\ell^{\rm pr}(\commutator{g_i,h_i})=\ell^{\rm rk}(\commutator{g_i,h_i})$ and by the end of Lemma~\ref{lem:char_cent_of_G/Nrk}, setting $h\defeq\seq{h_i}$, the commutator $\commutator{g,h}$ is in $N^{\rm rk}$ (and of course in $N$). This element has the same convergence type as $g$.
	\end{proof}
	
	\begin{remark}\label{rmk:ct_Nrk_onto_L}
		In view of Lemma~\ref{lem:ct_Npr_onto_L}, this implies that even the restriction $\rest{\ct}_{N^{\rm rk}}:N^{\rm rk}\to L$ is surjective.
	\end{remark}
	
	Now we are ready to justify the name \enquote*{associated order ideal} for $\ct(N)$ for $N$ normal in $G$ lying between $N_0$ and $N^{\rm pr}$.
	
	\begin{lemma}\label{lem:nor_subgrp_blw_Nrk}
		Let $N$ be a normal subgroup between $N_0$ and $N^{\rm pr}$. Then $\ct(N)$ is an order ideal of $L$. Moreover, the mappings
		$$
		\begin{tikzcd}
		\set{N\norsubgrpeq G}[N_0\subgrpeq N\subgrpeq N^{\rm rk}] \arrow[r,yshift=0.7ex,"\alpha"] & \set{\text{order ideals of }(L,\leq)} \arrow[l,yshift=-0.7ex,"\beta"] 
		\end{tikzcd}
		$$
		defined by $\alpha:N\mapsto\ct(N)$ and $\beta:I\mapsto\set{g\in N^{\rm rk}}[\ct(g)\in I]$ are isomorphisms of posets and mutually inverse to each other.
	\end{lemma}
	
	\begin{proof}
		By Lemma~\ref{lem:ct_Npr_to_Nrk}, we may restrict to the case $N\subgrpeq N^{\rm rk}$. Clearly, $\ct(N)$ is not empty as $1_G\in N$. If $r\in\ct(N)$, Lemma~\ref{lem:rk_bd_gen_class_grps} implies that if $s\leq r$ then $s\in\ct(N)$. As $(L,\leq)$ was a linear order, $\ct(N)$ is an order ideal.
		
		Concerning the second part: Obviously, both maps are inclusion preserving.
		To show that $\beta\circ\alpha:N\mapsto\ct(N)\mapsto\set{g\in N^{\rm rk}}[\ct(g)\in\ct(N)]$ is the identity, i.e.~that $\ct(g)\in\ct(N)$ iff $g\in N$, is just another straightforward application of Lemma~\ref{lem:rk_bd_gen_class_grps}. The map $\alpha\circ\beta:I\mapsto\set{g\in N^{\rm rk}}[\ct(g)\in I]\mapsto\set{\ct(g)\in L}[\ct(g)\in I]$ is the identity by Remark~\ref{rmk:ct_Nrk_onto_L}.
	\end{proof}
	
	\begin{remark}\label{rmk:prop_ct}
		Similarly, one can prove that for $g\in N^{\rm rk}$ and $S\subseteq N^{\rm rk}$ it holds that $\ct(g)$ is in the order ideal of $(L,\leq)$ generated by $\ct(S)$ if and only if $g\in\gennorsubgrp{S}N_0$.
	\end{remark}
	
	\begin{corollary}
		The normal subgroups of $G$ lying between $N_0$ and $N^{\rm rk}$ are linearly ordered.
	\end{corollary}
	
	\begin{proof}
		This holds by the correspondence of Lemma~\ref{lem:nor_subgrp_blw_Nrk} as $(L,\leq)$ is linear and so its ideals are linearly ordered by inclusion.
	\end{proof}
	
	\begin{corollary}\label{cor:blw_Nrk_perf}
		For $N\norsubgrpeq N^{\rm rk}$ containing $N_0$ it holds that $N$ is perfect, i.e.~$N=N'$.
	\end{corollary}
	
	\begin{proof}
		Pick $g\in N$ and proceed as in the proof of Lemma~\ref{lem:ct_Npr_to_Nrk}. The construction of $h$ shows that $\ct(h)\leq\ct(g)$, so Remark~\ref{rmk:prop_ct} implies $h\in\gennorsubgrp{g} N_0\subseteq N$. Hence the commutator $\commutator{g,h}$ lies in $N'$ and $\ct(\commutator{g,h})=\ct(g)$, so $\ct(N')=\ct(N)$.
		To apply Lemma~\ref{lem:nor_subgrp_blw_Nrk} to deduce that $N=N'$, we still need that $N_0\subgrpeq N'$. It is therefore enough to show that $N_0$ is perfect. This follows from the following simple application of Lemma~\ref{lem:bd_gen}: Let $h\in H$ be an element with $\ell^{\rm pr}(h)\geq\varepsilon$ for some fixed $\varepsilon>0$, where $H$ is one of the $H_i$. Then the previously mentioned result implies that $H'\supseteq(h^H)^{\ast k}h^{-k}=H$ for a fixed integer $k$ only depending on $\varepsilon$. Hence, if $g=\seq{g_i}\in N_0$, applying this in every component $i$ for which $g_i\neq 1_{H_i}$ yields that $g\in N_0'$ as wished.
	\end{proof}
	
	\begin{remark}
		The argument at the end of the proof also shows that $G$ itself is perfect.
	\end{remark}
	
	\begin{remark}\label{rmk:Nrk/N_0_char_sgrp}
		Especially, this implies that for $N\subgrpeq N^{\rm pr}$ it holds that $N\setmeet N^{\rm rk}=\commutator{G,N}= N'$ by the end of Lemma~\ref{lem:char_cent_of_G/Nrk} and Corollary~\ref{cor:blw_Nrk_perf}. Hence $N^{\rm rk}=(N^{\rm pr})'$ and so $N^{\rm rk}/N_0$ is characteristic in $G/N_0$ as well.
	\end{remark}
	
	Using the correspondence of Lemma~\ref{lem:nor_subgrp_blw_Nrk}, we introduce the normal subgroup $N_1$ as the normal subgroup of $N^{\rm rk}$ such that $\ct(N_1)=I_1$, where 
	$I_1\defeq\set{\nathom{\seq{r_i}}\in L}[n_ir_i<C\text{ for some } C>0\text{ along }\mathcal U]$.
	It is obvious that $I_1$ covers the trivial ideal $I_0=\set{\nathom{\seq{0}}}$ corresponding to $N_0$, i.e.~there is no ideal properly between them. Hence $N_1$ covers $N_0$. Let us mention here that for $M,N$ normal subgroups of $G$ between $N_0$ and $N^{\rm rk}$ it holds that $M$ is covered by $N$ if and only if the corresponding order ideals $I$ and $J$ of $(L,\leq)$ are of the form $I=\set{r\in L}[r<s]$ and $J=\set{r\in L}[r\leq s]$ for some $s\neq\nathom{\seq{0}}$.
	
	Because we will need them later, we introduce the normal subgroups $A_0\defeq\set{g \in N^{\rm pr}}[\ct(g)\in I_0]$ and $A_1\defeq\set{g \in N^{\rm pr}}[\ct(g)\in I_1]$.
	
	Now we complete the picture of the lattice of normal subgroups of $G/N_0$ by looking at an arbitrary $N\norsubgrpeq N^{\rm pr}$ with $N_0\subgrpeq N$.
	
	Let us at first assume that $N_1\subgrpeq N$. Then the following lemma applies:

	\begin{lemma}\label{lem:lat_abv_N_1}
		Define mappings between the following two sets
		$$
		\begin{tikzcd}
			& \set{N\norsubgrpeq G}[N_1\subgrpeq N\subgrpeq N^{\rm pr}] \arrow[d,xshift=0.7ex,"\Phi"] &\\
			& \set{M}[M\norsubgrpeq G,\ N_1\subgrpeq M\subgrpeq N^{\rm rk}]\times\set{A}[A\norsubgrpeq G,\ N_1\subgrpeq A\subgrpeq A_1] \arrow[u,xshift=-0.7ex,"\Psi"] &
		\end{tikzcd}
		$$
		by $\Phi:N\mapsto(N\setmeet N^{\rm rk}, N\setmeet A_1)$ and $\Psi:(M,A)\mapsto MA$.	Then $\Phi$ and $\Psi$ are isomorphisms of posets and mutually inverse to each other.
	\end{lemma}
	
	\begin{proof}
		Obviously, both maps are inclusion preserving. Moreover, $\Phi\circ\Psi(M,A)=\Phi(MA)=(MA\setmeet N^{\rm rk},MA\setmeet A_1)=(M(A\setmeet N^{\rm rk}),A(M\setmeet A_1)) $ by Dedekind's modular law. However, it is easy to see from the definitions that $A\setmeet N^{\rm rk}=M\setmeet A_1=N_1$, so $\Phi\circ\Psi=\id$ as $N_1\subgrpeq M,A$.
		
		Similarly, we compute $\Psi\circ\Phi(N)=\Psi(N\setmeet N^{\rm rk}, N\setmeet A_1)=(N\setmeet N^{\rm rk})(N\setmeet A_1)\subseteq N$. So it is enough to show that every element $g=\seq{g_i}\in N$ can be written as a product of elements from $N\setmeet N^{\rm rk}$ and $N\setmeet A_1$. Choose $\lambda_i$ ($i\in I$) such that $\ell^{\rm rk}(\lambda_i^{-1}g)=\ell^{\rm pr}(g_i)$. Let $h_i\in H_i$ be an element which has all but two diagonal entries equal to $\lambda_i$ (which exists by Lemma~\ref{lem:all_qsi_scal_attained}). Then setting $f\defeq\seq{g_ih_i^{-1}}$ and $h\defeq\seq{h_i}$, $g=fh$ is the desired product decomposition.
	\end{proof}
	
	\begin{remark}
		Essentially, Lemma~\ref{lem:lat_abv_N_1} (in combination with Lemma~\ref{lem:nor_subgrp_blw_Nrk}) says that the lattice of normal subgroups of $G/N_1$ is isomorphic to the product of the linear order of ideals of $(L,\leq)$ different from $I_0$ and the subgroup lattice of the abelian group $Z\iso N^{\rm pr}/N^{\rm rk}=N^{\rm rk}A_1/N^{\rm rk}\iso A_1/(A_1\setmeet N^{\rm rk})= A_1/N_1$ from Lemma~\ref{lem:char_cent_of_G/Nrk}. If $\Phi:N\mapsto(M,A)$, then $N\mapsto(\ct(M),\varphi(A))$ is the described isomorphism of lattices, where $\varphi$ is the map from Lemma~\ref{lem:char_cent_of_G/Nrk}.
		
		In this situation, $M=N'$ is the abelianization (by Remark~\ref{rmk:Nrk/N_0_char_sgrp} above) and $\varphi(A)=\varphi(MA)=\varphi(N)\iso N N^{\rm rk}/N^{\rm rk}\iso N/(N\setmeet N^{\rm rk})=N/N'$ is the universal abelian quotient.
	\end{remark}
	
	Now assume that $N_1\not\subgrpeq N$. Then it follows from Lemma~\ref{lem:ct_Npr_to_Nrk} that $\ct(N)=\ct(N\setmeet N^{\rm rk})=\ct(N_0)=\set{\nathom{\seq{0}}}=I_0$ (as by assumption $N_0\subgrpeq N$). Hence we have established the following result.
	
	\begin{lemma}\label{lem:lat_not_abv_N_1}
		If $N$ is normal in $G$ containing $N_0$ and $N_1\not\subgrpeq N$, then $N\subgrpeq A_0$.
	\end{lemma}
	
	Finally, we describe when a normal subgroup as in Lemma~\ref{lem:lat_not_abv_N_1} is contained in a normal subgroup as in Lemma~\ref{lem:lat_abv_N_1}:
	
	\begin{lemma}
		A normal subgroup $N$ containing $N_0$ but not $N_1$ is contained in a normal subgroup $K\subgrpeq N^{\rm pr}$ containing $N_1$ if and only if $N\subgrpeq A$ (or equivalently $\varphi(N)\subgrpeq\varphi(A)$, with $\varphi$ the map from Lemma~\ref{lem:char_cent_of_G/Nrk}), where $\Phi(K)=(M,A)$ is from Lemma~\ref{lem:lat_abv_N_1}.
	\end{lemma}
	
	\begin{proof}
		Assume $N\subgrpeq K$. Then by Lemma~\ref{lem:lat_not_abv_N_1} it holds that $N\subgrpeq K\setmeet A_0\subgrpeq K\setmeet A_1=A$ implying $\varphi(N)\subgrpeq\varphi(A)$. Conversely, by Lemma~\ref{lem:char_cent_of_G/Nrk}, the assumption $\varphi(N)\subgrpeq\varphi(A)$ implies $NN^{\rm rk}\subgrpeq AN^{\rm rk}$. Thus $NN^{\rm rk}\setmeet A_1=N(N^{\rm rk}\setmeet A_1)=NN_1\subgrpeq AN^{\rm rk}\setmeet A_1=A(N^{\rm rk}\setmeet A_1)=AN_1=A$ by Dedekind's modular law. So $N\subgrpeq A\subgrpeq K$, completing the proof.
	\end{proof}
	
	We will end the article with a few remarks.

	\begin{remark} \label{end}
		Again, we have not yet covered the case that the $H_i$ ($i\in I$) are double covers of $\Omega^\pm_{2m}(q)$ or $\Omega_{2m+1}(q)$ ($q$ odd). In this case, define $N_0$ as above and let $A_0$ be the elements of $G=\prod_i{H_i}$ which are central along the ultrafilter $\mathcal U$. Then $G/N_0$ is a twofold cover of an ultraproduct $G/M_0$ of orthogonal groups with kernel $M_0/N_0\iso C_2$.

		It follows from Remark~\ref{rmk:orth_case_dbl_cover_rel_bd_gen} that if $N_0\subgrpeq N\norsubgrpeq G$ contains an element non-central modulo $N_0$, i.e.~it is not contained in $A_0$, then $M_0\subgrpeq N$, so it corresponds to a normal subgroup of $G/M_0$, which we have described.
		
		Finally, when $N_0\subgrpeq N\subgrpeq A_0$ with $M_0\not\subgrpeq N$ and $M_0\subgrpeq K\subgrpeq G$, then $N\subgrpeq K$ iff $NM_0\subgrpeq K$ and this again can be decided by considering the lattice of normal subgroups of $G/M_0$, which we discussed.
		
		Using the same argument, one can show that the normal subgroups of an ultraproduct of the double covers of simple alternating groups are still linearly ordered by inclusion (since here $A_0=M_0$).
	\end{remark}
	
	\section*{Acknowledgments}
	
	The results of this article are part of the PhD thesis of the first author. This research was supported by ERC Consolidator Grant No.\ 681207.

\end{document}